\documentclass[12pt]{article}%
\usepackage{bm,graphics,epsfig,color,float,latexsym,array,amsthm,mathrsfs,amssymb,amsfonts,graphicx,amsmath,psfrag,enumerate}%
\usepackage{epic,fancybox}
\numberwithin{equation}{section}
\newtheorem{thm}{Theorem}[section]
\newtheorem{fact}[thm]{Fact}
\newtheorem{prop}[thm]{Proposition}
\newtheorem{lem}[thm]{Lemma}
\newtheorem{defn}[thm]{Definition}
\newtheorem{cor}[thm]{Corollary}

\newtheorem{conj}[thm]{Conjecture}
\newtheorem{ex}[thm]{Example}
\newtheorem{obs}[thm]{Observation}

\newcommand{\sinv}{\piinv_\text{b}}
\newcommand{\s}{\pi_\text{b}}
\newcommand{\p}{\mathcal{P}}
\newcommand{\Z}{\mathbb{Z}}
\newcommand{\ra}{\rightarrow}
\newcommand{\wt}{\tilde}
\newcommand{\B}{\mathcal{B}}
\newcommand{\Balg}{\mathscr{B}}
\newcommand{\F}{\mathscr{F}}
\newcommand{\C}{\mathscr{C}}
\newcommand{\A}{\mathcal{A}}
\newcommand{\piinv}{\pi^{-1}}
\newcommand{\E}{\mathbb{E}}

\newcommand{\pGcom}{\partial G^c}

\newcommand {\DS} [1] {${\displaystyle #1}$}

\begin{document}
\date{}
\title{Class Degree and Relative Maximal Entropy}
\author{Mahsa Allahbakhshi and Anthony Quas\thanks{
We thank the referee for detailed and helpful comments.}}
%

\maketitle
\begin{abstract}
Given a factor code $\pi$ from a one-dimensional shift of finite type $X$ onto an irreducible sofic shift $Y$, if $\pi$ is finite-to-one there is an invariant called the degree of $\pi$ which is defined the number of preimages of a typical point in $Y$. We generalize the notion of the degree to the class degree which is defined for any factor code on a one-dimensional shift of finite type. Given an ergodic measure $\nu$ on $Y$, we find an invariant upper bound on the number of ergodic measures on $X$ which project to $\nu$ and have maximal entropy among all measures in the fibre $\pi^{-1}\{\nu\}$. We show  that this bound and the class degree of the code agree when $\nu$ is ergodic and fully supported. One of the main ingredients of the proof is a uniform distribution property for ergodic measures of relative maximal entropy.
\end{abstract} 
\section{Introduction}
It is a well-known result that a 1-dimensional irreducible shift of finite type on a finite alphabet has a unique measure of
maximal entropy, the measure developed by Shannon and Parry \cite{parry,shannon}. In contrast, we consider the relative case in which one is given a factor code $\pi: X\rightarrow Y$ from a shift of finite type $X$ to a sofic shift $Y$, and a measure $\nu$ on $Y$. In this case ergodic measures on $X$ in the fibre $\pi^{-1}\{\nu\}$ having maximal entropy in the fibre, so-called measures of relative maximal entropy, are not well understood.

Measures of relative maximal entropy appear frequently in different topics in symbolic dynamics, and have been applied to the problem of computing the Hausdorff dimension of certain sets by Gatzouras and Peres~\cite{GatPerexp}. There are connections of measures of relative maximal entropy with functions of Markov chains \cite{black,BoyTun,BurRos,MarPetWill}, measures that maximize a weighted entropy functional \cite{GatPer,shin}, 
the theory of pressure and equilibrium states \cite{isr,kel,ruelle}, relative pressure and relative equilibrium states \cite{LedWal,LedYo,wal}, and compensation functions \cite{BoyTun,wal}. Other uses of such measures arise from their application in the mathematics of information transfer \cite{karlinfo} and information-compressing channels \cite{MarPetWill}.

Unlike in the case considered by Shannon and Parry where the measure of maximal entropy is unique, it is known that measures of relative maximal entropy may not be 
unique~\cite{karlinfo,pqs}. It is natural to ask how many of such measures there may be for a particular factor code $\pi:X\to Y$ and an ergodic invariant measure on Y. A previous result on this problem is the result of Petersen, Quas, and Shin \cite{pqs} which established an explicit finite upper bound on the number of measures of relative maximal entropy. They considered the case (always achievable by recoding) that $\pi$ is a 1-block code from a 1-step shift of finite type. In this case, the main result of~\cite{pqs} states that the number of measures of relative maximal entropy is bounded above by the number of preimages of any symbol $w$ of $Y$ which has positive probability with respect to the measure on $Y$.

The bound in~\cite{pqs} suffers from not being invariant under conjugacy and becomes arbitrarily large simply by recoding to a higher block presentation. To avoid this issue one possibility is to take the minimum of this bound over all shifts of finite type $\wt X$ which are conjugate to
$X$. However, no algorithm for computing this minimum has been found. On the other hand, in this paper we find a simpler conjugacy-invariant upper bound which can be strictly better and is never worse than the minimum-over-conjugates bound mentioned above.

Given a finite-to-one factor code $\pi$ from a shift of finite type $X$ to a sofic shift $Y$, the \textsf{degree} is defined to be the minimal cardinality of the set of preimages of a point in $Y$. This is widely-studied and known to be invariant under recoding~\cite{lm}. One can show that the number of preimages of every transitive point in $Y$ is exactly the degree of $\pi$. Previous efforts to find an analogue of this when $\pi$ is infinite-to-one have proved unsuccessful~\cite{boyle05}.

Given $y$ in $Y$, we define an equivalence relation on $X$ which is respected by recoding. The relation is motivated by communicating classes in Markov chains. Loosely, two points $x$ and $x'$ in $\piinv(y)$ lie in the same equivalence class ``transition class'' if one can find an element of $\piinv(y)$ matching either of $x$ and $x'$ from $-\infty$ to an arbitrarily 
large positive coordinate $n$ and matching the other at 
$+\infty$. The minimal number of transition classes (always finite) in $\piinv(y)$ when $y$ runs over the points in $Y$ is called the \textsf{class degree} of the code. We show that the number of transition classes over every transitive point of $Y$ is exactly the class degree. If $\pi$ is finite-to-one the degree and the class degree of $\pi$ agree.

We prove an analogue of Theorem 1 in~\cite{pqs}: given an ergodic measure $\nu$ on $Y$, for distinct ergodic measures $\mu_1$ and $\mu_2$ of relative maximal entropy, the relative independent joining of $\mu_1$ and $\mu_2$ assigns measure zero to the set of points $(u,v)$ such that $u$,\,$v$ are in a common transition class.

Using a combinatorial characterization of class degree and further arguments, we show that the number of transition classes above a typical point of $\nu$ is an upper bound on the number of ergodic measures of maximal entropy in the fibre $\piinv\{\nu\}$. This bound is equal to the class degree of the code when $\nu$ is ergodic and fully supported.

One key ingredient to this work is a relative version of the uniform conditional distribution property for measures of maximal entropy. We follow techniques developed by Burton and Steif~\cite{bs} (which were related to an earlier work of Lanford and Ruelle~\cite{LanfRuel69}) to show a uniform conditional distribution property for measures of relative maximal entropy.

\section{Background}
Throughout the paper, $X$ will denote a one-dimensional shift of finite type on a finite alphabet (SFT) and the shift map will be denoted by $T_X$, or simply by $T$ when there is no ambiguity. We will deal with a factor code $\pi$ from $X$ to a sofic shift $Y$. A detailed description of these notions is given in~\cite{lm}. A triple $(X,Y,\pi)$ is called a \textsf{factor triple} when $\pi:X\ra Y$ is a factor code from a SFT $X$ to a sofic shift $Y$. The alphabet of a shift space $X$ is denoted by $\A(X)$ and the $\sigma$-algebra on $X$ generated by cylinder sets of $X$ is denoted by $\Balg_X$. The set of all $n$-blocks that occur in points of $X$ is denoted by $\B_n(X)$, and the language of $X$, $\mathscr L(X)$, is the collection $\bigcup_{n=0}^{\infty}\B_n(X)$. Let $x$ be in $X$ and $G\subseteq\Z$. The configuration which occurs in $x$ on $G$ is denoted by $x_G$. If $G=[i,j]$ is a $\Z$-interval we sometimes denote $x_G$ by $x_{[i,j]}$. We say two factor triples $(X,Y,\pi)$ and $(\wt X,\wt Y,\wt\pi)$ are \textsf{conjugate}, and denote it by $(X,Y,\pi)\cong(\wt X,\wt Y,\wt\pi)$, if $X$ is conjugate to $\wt X$ under a conjugacy $\phi$, $Y$ is conjugate to $\wt Y$ under a conjugacy $\psi$, and $\wt\pi\circ\phi=\psi\circ\pi$. Let $(X,Y,\pi)$ be a factor triple where $\pi$ is a 1-block factor code induced by the map $\s:\A(X)\to\A(Y)$ (b stands for block). The map $\s$ naturally extends to elements of $\mathscr L(X)$. Above every $Y$-block $W$ of length $n$
there is a set of $X$-blocks $W'$ of length $n$ which are sent to $W$ by
$\s$; i.e., $\s(W')=W$. Given $0\leq i \leq n-1$, define
$$d(W,i)=|\{a\in\A(X)\colon \exists W' \text{ with }\s(W')=W,\,
W'_i=a\}|,$$ and let $$d^*_{\pi}=\min\{d(W,i):W\in\mathscr L(Y),\, 0\leq
i\leq |W|-1\}.$$ A \textsf{magic block} is a block $W$ such that
$d(W,i)=d^*_{\pi}$  for some $0\leq i\leq |W|-1$. Such an index $i$ is
called a \textsf{magic coordinate} of $W$. A factor code $\pi$ has a
\textsf{magic symbol} if there is a magic block of $\pi$ of length $1$.
\begin{thm}~\cite[Theorem 9.1.7]{lm}\label{thm:magicsymbol}
 Let $(X,Y,\pi)$ be a factor triple. There is a factor
triple $(\wt{X},\wt{Y},\wt\pi)$ conjugate to
$(X,Y,\pi)$ such that $\tilde\pi$ is a 1-block code with a magic symbol.
\end{thm}

\section {Uniform Conditional  Distribution}\label{sec:unif}
By a well-known result of Parry \cite{parry}, generalizing an earlier result of Shannon \cite{shannon}, in one dimension every irreducible shift of finite type on a finite
alphabet has a unique measure of maximal entropy. Burton and Steif
\cite{bs} give a counterexample to this statement in higher dimensions. However, they show that such measures all have the uniform conditional distribution property stated in Theorem~\ref{BuSt} (which is related to an earlier work of Lanford and Ruelle~\cite{LanfRuel69}). Given a finite set $G\subseteq\mathbb{Z}^d$, the boundary of the complement of $G$ is $\partial G^{c}=\{i\in G^{c}\colon \exists j\in G \text{ with } \|i-j\|=1\}$.
\begin{thm}\cite[Proposition 1.19]{bs}\label{BuSt}
Let $X$ be a 1-step d-dimensional SFT and $\mu$ be a measure of maximal entropy on $X$. Then
the conditional distribution of $\mu$ on any finite set $G\subseteq\Z^d$ given
the configuration on $G^c$ is $\mu$-a.s. uniform over all configurations
on $G$ which extend the configuration on $\pGcom$.
\end{thm}

Given a factor triple $(X,Y,\pi)$ and an ergodic measure $\nu$ on $Y$, there can exist
more than one ergodic measure of relative maximal entropy over $\nu$; i.e., there can be more than one ergodic measure on $X$ which projects to $\nu$ and has maximal entropy among all measures in the fibre $\piinv\{\nu\}$, see
\cite[Example 3.3]{pqs}. We use Lemma~\ref{lem:limit} and follow techniques developed by Burton and Steif in the proof of Theorem~\ref{BuSt} to show the uniform conditional distribution property for measures of relative maximal entropy in Theorem~\ref{CUniformDist}.

\begin{lem}\label{lem:limit} \cite[Theorem 4.7]{walters}
Let $(X,\Balg,\mu)$ be a probability space. Let $\A$ be a finite sub-algebra of $\Balg$ and let $(\F_n)_{n=1}^\infty$ be an increasing sequence of sub-$\sigma$-algebras of $\Balg$ with $\bigvee_{n=1}^\infty\F_n=\F$. Then $H(\A|\F_n)\to H(\A|\F)$.
\end{lem}

\begin{thm}\label{CUniformDist}
Let $\pi:X\to Y$ be a 1-block factor code from a 1-step, 1-dimensional SFT $X$ to a
 sofic shift $Y$. Let $\nu$ be an invariant measure on $Y$, and $\mu$ be an
invariant measure of relative maximal entropy over $\nu$. Given a finite set $G\subseteq\Z$ and $y$ in $Y$ the
conditional distribution of $\mu$ on $\piinv(y)$ restricted to $G$ given
the configuration on $G^c$ is $\mu$-a.s. uniform over all configurations
on $G$ which extend the configuration on $\pGcom$ and map to the same
configuration in $Y$ under the factor code $\pi$.
\end{thm}

\begin{proof}
Let $G$ be a finite subset of $\Z$. Let $\Delta$ be a configuration of
$Y$ on $G$. Pick a configuration $\eta$ of $X$  on $\pGcom$ such that $\mu(\eta\cap\piinv(\Delta))>0$. Starting from $\mu$, we define a measure $\tilde\gamma$ on $X$ by uniformizing over pre-images of $\Delta$ that have $\eta$ on the boundary. We then show that if $\mu$ does not have the required uniform conditional distribution property then $\tilde\gamma$ has greater entropy than $\mu$, but still is an element of the fibre $\piinv\{\nu\}$. This is a contradiction and will therefore establish the required uniform conditional distribution property of $\mu$.

Let $D=\{\alpha_1,\dots,\alpha_{L}\}$ be the set of all configurations on $G$
which extend $\eta$ and map to $\Delta$ under the factor code $\pi$. Let the $\Z$-interval 
$R=[-m,m-1]$ be large enough so that $G\cup \partial G^c\subseteq
R$ and let $\left(R_n\right)_{n\in\Z}$ be the partition of $\Z$ by translates of $R$ ($R_0=R$), i.e. $R_n=[(2n-1)m,(2n+1)m-1]$. Let $G_n$ and $\partial G^c_n$ be the corresponding translates of $G$ and
$\partial G^c$ in $R_n$. Given $S\subseteq\Z$ let $\p(X_S)$ be the partition of $X$ generated by the configurations of $X$ on $S$, and $\sigma(X_S)$ be the $\sigma$-algebra generated by $\p(X_S)$. When $S=[a,b]$ is a $\Z$-interval we sometimes denote $X_S$ by $X_a^b$. Considering $R$ above, $\sigma(X_{-m}^{m-1})$ is the finite
$\sigma$-algebra generated by the partition
$$
\p(X_{-m}^{m-1})=\{\left._{-m}[x_{-m}x_{-m+1}\dots x_{m-1}]\right._{m-1}:
x_{-m}x_{-m+1}\dots x_{m-1}\in\mathscr L(X)\},
$$
and $\sigma(X_{-\infty}^{-m-1})=
\sigma(X^{-m-1}_{-3m})\vee\sigma(X^{-3m-1}_{-5m})\vee\dots$.

We obtain the measure $\gamma$ from the measures $\mu$ and $\eta$ as follows.
Define the map $\Phi:X\times D^{\Z}\rightarrow X$ by
\begin{equation*}
\Phi(x,\zeta)_{G_n}=
\begin{cases}
{\zeta_n }& \textrm{ if $x_{\partial G^c_n}=\eta$ and 
$x_{G_n}\in D$}
\\
x_{G_n}& \textrm  { otherwise,}
\end{cases}
\end{equation*}
and $\Phi(x,\zeta)_{G^c_n\cap R_n}=x_{G^c_n\cap R_n}$ for each integer $n$. Since $\zeta_n$ is in $D$ and each element of $D$ extends $\eta$, the
assumption that $X$ is a 1-step SFT implies that $\Phi(x,\zeta)$ lies in $X$. For each
$\zeta$ in $D^{\Z}$ we have $\pi(x)=\pi(\Phi(x,\zeta))$ since
$\Phi(x,\zeta)$ and $x$ are the same except having alternative
$\alpha_i$'s in the same positions. Let $C$ be a cylinder set of $X$. Define
$\gamma(C)=(\mu\times\lambda)\Phi^{-1}(C)$ where $\lambda$ is the
Bernoulli $(1/L,\dots,1/L)$ measure on $D^{\Z}$. The measure $\gamma$ is not necessarily invariant under $T$;
however, for each cylinder set $C$ we have
$\gamma(C)=\gamma(T^{-2m}C)$. So the new measure $\tilde\gamma$ on $X$ defined by 
$\tilde\gamma(C)=\frac{1}{2m}\left(\gamma(C)+\dots+\gamma(T^{-2m+1}C)\right)$ is $T$-invariant. Since for each cylinder set $E$ of $Y$ we have
$$\gamma(\piinv E)=(\mu\times\lambda)(\piinv E\times D^{\Z})=\nu(E),$$ we deduce that both measures $\gamma$ and $\tilde\gamma$ are in
the fibre $\piinv\{\nu\}.$

Define an equivalence relation on $X$ as follows. Given $x,x'$ in $X$,
say $x\sim_{0} x'$ if either $x_R=x'_R$ or else $x_{R\cap G^c}=x'_{R\cap
G^c}$, $x_{\pGcom}=x'_{\pGcom}=\eta$, and $x_G,x'_G$ in $D$. Denote the
equivalence class containing $x$ by $C_0(x)$. Such equivalence
classes form a sub-partition of $\p(X_{-m}^{m-1})$. Let $\A$ be the
$\sigma$-algebra generated by these equivalence classes. We show $h_{\tilde\gamma}(T)\geq h_\mu(T)$ as follows, using Lemma~\ref{lem:parts} which appears below.
\begin{align}\label{eq:maxentropy}
h_{\tilde\gamma}\left(T\right)&=
\frac{1}{2m}h_{\tilde\gamma}\left(T^{2m}\right)
\\
&=\frac{1}{2m}h_\gamma\left(T^{2m}\right)\notag
\\
&=\frac{1}{2m}H_{\gamma}\left(\sigma(X_{-m}^{m-1})|\sigma(X_{-\infty}^{-m-1})\right)\notag
\\
&=\frac{1}{2m}H_{\gamma}\left(\mathcal A|\sigma(X_{-\infty}^{-m-1})\right)+
\frac{1}{2m}H_{\gamma}\left(\sigma(X_{-m}^{m-1})|
\sigma(X_{-\infty}^{-m-1})\vee\mathcal A\right)\notag
\\
&=\frac{1}{2m}H_{\gamma}\left(\mathcal A|\sigma(X_{-\infty}^{-m-1})\right)
+\frac{1}{2m}H_{\gamma}\left(\sigma(X_{-m}^{m-1})|\mathcal A\right)&\text{Lem~\ref{lem:parts}}(a)\notag
\\
&\geq\frac{1}{2m}H_{\gamma}\left(\mathcal
A|\sigma(X_{-\infty}^{-m-1})\right)+
\frac{1}{2m}H_{\mu}\left(\sigma(X_{-m}^{m-1})|\mathcal
A\right)&\text{Lem~\ref{lem:parts}}(b)\notag
\\
&\geq \frac{1}{2m}H_{\mu}\left(\mathcal A|\sigma(X_{-\infty}^{-m-1})\right)
+\frac{1}{2m}H_{\mu}\left(\sigma(X_{-m}^{m-1})|\mathcal A\right)&\text{Lem~\ref{lem:parts}}(c)\notag
\\
&\geq\frac{1}{2m}H_{\mu}\left(\mathcal A|\sigma(X_{-\infty}^{-m-1})\right)
+\frac{1}{2m}H_{\mu}\left(\sigma(X_{-m}^{m-1})|
\sigma(X_{-\infty}^{-m-1})\vee \mathcal A\right)\notag
\\
&=\frac{1}{2m}H_{\mu}\left(\sigma(X_{-m}^{m-1})|\sigma(X_{-\infty}^{-m-1})\right)\notag
\\
&=\frac{1}{2m} h_{\mu}\left(T^{2m}\right)\notag
\\
&=h_{\mu}\left(T\right).\notag
\end{align}
\begin{lem}\label{lem:parts} Reusing previous notations, we have
\begin{enumerate}[(a)]
\item $H_{\gamma}\left(\sigma(X_{-m}^{m-1})|
\sigma(X_{-\infty}^{-m-1})\vee\mathcal A\right)=H_{\gamma}\left(\sigma(X_{-m}^{m-1})|
\mathcal A\right)$.
\item $H_{\gamma}\left(\sigma(X_{-m}^{m-1})|
\mathcal A\right)\geq H_{\mu}\left(\sigma(X_{-m}^{m-1})|
\mathcal A\right)$. Equality occurs if and only if $$\frac{\mu(\alpha\cap \bar A)}{\mu(\bar A)}=1/L,$$ for each $\alpha$ in $D$ and $\bar A$ in $\p(\A)$ such that for each $x$ in $\bar\A$ we have $x_{\pGcom}=\eta$ and $x_G$ in $D$.
\item $H_{\gamma}\left(
\mathcal A|\sigma(X_{-\infty}^{-m-1})\right)\geq H_{\mu}\left(
\mathcal A|\sigma(X_{-\infty}^{-m-1})\right)$.
\end{enumerate}
\end{lem}
\begin{proof}
By definition, for a measure $\rho$ in $\{\mu,\gamma\}$ we have
$$
H_{\rho}\left(\sigma(X_{-m}^{m-1})|\mathcal A\right)
=-\sum_{i,k}\rho(O_i\cap A_k)\log\frac{\rho(O_i\cap A_k)}
{\rho(A_k)}
$$
where $O_i$ is in $\p\left(X_{-m}^{m-1}\right)$ and $A_k$ is in $\p(\A)$ (if $\rho(O_i\cap A_k)$ is zero we define $\rho(O_i\cap A_k)\log\frac{\rho(O_i\cap A_k)}{\rho(A_k)}=0$). 
Let $x$ be in $A_k$. If $x_{\partial G^c}\not=\eta$ or $x_G\notin D$ then
for every $O_i$ in $\p\left(X_{-m}^{m-1}\right)$ we have either $O_i\cap
A_k=\emptyset$ or $O_i\cap A_k=A_k$ which both imply $\rho(O_i\cap
A_k)\log\frac{\rho(O_i\cap A_k)}{\rho(A_k)}=0$. Let $\{\bar
A_1,\dots,\bar
A_M\}$ be the set consisting of elements in the partition $\p(\A)$ so that for any point $x$ in $\bar {A}_k$ we have $x_{\partial G^c}=\eta$ and $x_G$ in $D$. Then
$$
H_{\rho}(\sigma(X_{-m}^{m-1})|\mathcal A)=-\sum_{i,k}\rho(O_i\cap \bar
A_k)\log\frac{\rho(O_i\cap \bar A_k)}{\rho(\bar A_k)}.
$$
There are exactly $L$ pairwise disjoint sets $O_i$ in $\p\left(X_{-m}^{m-1}\right)$ defined by blocks which agree everywhere except on $G$, and form a partition of $\bar A_k$. Let these sets
be denoted by $O_{k,1},\dots,O_{k,L}$ where $O_{k,i}=\alpha_i\cap \bar A_k$ for each $1\leq i\leq L$. It follows that 
\begin{equation}\label{eq:rho}
\begin{split}
H_{\rho}(\sigma(X_{-m}^{m-1})|\mathcal A)
&=-\sum_{k,i}\rho(O_{k,i}\cap \bar A_k)
\log\frac{\rho(O_{k,i}\cap \bar A_k)}{\rho(\bar A_k)}
\\
&=-\sum_{k,i}\rho(\alpha_i\cap \bar A_k)
\log\frac{\rho(\alpha_i\cap \bar A_k)}{\rho(\bar A_k)}.
\end{split}
\end{equation}
Let $n\geq 1$. By definition of $\gamma$, for each configuration $\alpha$ in $D$, each cylinder set $P$ in $\p\left(X_{-(2n+1)m}^{-m-1}\right)$, and each $1\leq k\leq M$ we have $$\frac{\gamma(\alpha\cap P\cap \bar A_k)}{\gamma(P\cap \bar A_k)}=\frac{\gamma(\alpha\cap \bar A_k)}{\gamma(\bar A_k)}=\frac{1}{L}.$$ 
It follows that
\begin{equation*}
\begin{split}
H_{\gamma}\left(\sigma\left(X_{-m}^{m-1}\right)|
\sigma\left(X_{-(2n+1)m}^{-m-1}\right)\vee\mathcal A\right)&=-\sum_{i,j,k}\gamma(\alpha_i\cap P_j\cap \bar A_k)
\log\frac{\gamma(\alpha_i\cap P_j\cap \bar A_k)}{\gamma(P_j\cap \bar A_k)}
\\
&=\log L \sum_{i,k}\gamma(\alpha_i\cap \bar A_k)
\\
&=H_{\gamma}\left(\sigma\left(X_{-m}^{m-1}\right)|
\mathcal A\right)
\end{split}
\end{equation*}
(use the same argument we had before Equation \eqref {eq:rho} to get the first equality above).
Then (a) follows from Lemma~\ref{lem:limit} and the fact that $\left(\sigma\left(X_{-(2n+1)m}^{-m-1}\right)\right)_{n=1}^\infty$ is an increasing sequence of $\sigma$-algebras with $\bigvee_{n=1}^\infty \sigma\left(X_{-(2n+1)m}^{-m-1}\right)=\sigma\left(X_{-\infty}^{-m-1}\right)$. 

(b) is a standard property of entropy.

To show (c) let $\mathcal A_{-(2n+1)m}^{-m-1}$ denote the $\sigma$-algebra
$\bigvee_{i=1}^n T^{-2mi}\mathcal A$ (the choice of notation is
intended to remind the reader that $\mathcal A_{-\infty}^{-m-1}$ is a
sub-$\sigma$-algebra of $\sigma(X_{-\infty}^{-m-1}$).

Any set $B$ in $\mathcal A_{-(2n+1)m}^{-m-1}$ splits up into $L^K$
elements of $\sigma(X_{-(2n+1)m}^{-m-1})$, $B_1,\ldots, B_{L^K}$ of
equal measure (where $K$ is the number of occurrences in $B$ of blocks
that are randomized under $\gamma$). Each of these sets has the
property that $\gamma(A\cap B_i)/\gamma(B_i)=\gamma(A\cap
B)/\gamma(B)$. Accordingly we see that $H_\gamma(\mathcal A|\mathcal
A_{-(2n+1)m}^{-m-1})=H_\gamma(\mathcal
A|\sigma(X_{-(2n+1)m}^{-m-1}))$. Taking a limit we obtain
\begin{equation*}
H_\gamma(\mathcal A|\sigma(X_{-\infty}^{-m-1}))=H_\gamma(\mathcal
A|\mathcal A_{-\infty}^{-m-1}).
\end{equation*}

We then notice that $H_\gamma(\mathcal A|\mathcal
A_{-\infty}^{-m-1})=H_\mu(\mathcal A|\mathcal A_{-\infty}^{-m-1})$
since the measures $\mu$ and $\gamma$ agree on elements of
$\bigvee_{i=0}^\infty T^{-2mi}\mathcal A$.

Finally it is clear that $H_\mu(\mathcal A|\mathcal
A_{-\infty}^{-m-1})\ge H_\mu(\mathcal A|\sigma(X_{-\infty}^{-m-1}))$.
Combining these three equalities and inequalities we obtain
$H_\gamma(\mathcal A|\sigma(X_{-\infty}^{-m-1}))\ge H_\mu(\mathcal
A|\sigma(X_{-\infty}^{-m-1}))$ as required.
\end{proof}

\noindent {\em Proof of Theorem~\ref{CUniformDist} (continued). }Since $\mu$ is a measure of relative maximal entropy it follows that all of the inequalities in Equations~\eqref{eq:maxentropy} are forced
to be equalities. In particular, we have
$H_{\mu}\left(\sigma\left(X_{-m}^{m-1}\right)|\A\right)=H_{\gamma}\left(\sigma\left(X_{-m}^{m-1}\right)|\A\right)$. Then Lemma~\ref{lem:parts}(b) implies that
\begin{align}\label{eq:unif}
\frac{\mu\left(\alpha\cap \bar A\right)}{\mu(\bar A)}=1/L,
\end{align}
for each configuration $\alpha$ in $D$ and $\bar A$ in $\p(\A)$ such that for each  $x$ in $\bar A$ we have $x_{\pGcom}=\eta$ and $x_G$ in $D$. Note that given a finite set $G\subseteq\Z$, both configurations $\Delta\in\Balg_Y$ on $G$ and $\eta\in\Balg_X$ on $\pGcom$ with $\mu(\eta\cap\piinv(\Delta))>0$ are chosen arbitrarily. By choosing different configurations and noting that $\bar
A\in\p\left(\pi^{-1}(\sigma(Y_G))\bigvee \sigma(X_{R\cap
G^c})\right)=\p\left(\pi^{-1}(\sigma(Y_R))\bigvee \sigma(X_{R\cap
G^c})\right)$, Equation \eqref{eq:unif} implies that for any configuration
$\alpha$ of $X$ occurring at $G$ we have
\begin{equation*}
\E\left(1_{\alpha}| \pi^{-1}(\sigma(Y_R))\vee \sigma(X_{R\cap G^c})\right)(x)=
\begin{cases}
 \frac{1}{L_{(x,G)}}
& \alpha\textrm{ extends }x_{\partial G^c},\, x_G\in\piinv_\text{b}(\s(\alpha))
\\
0&\textrm{otherwise.  }
\end{cases}
\end{equation*}
where $L_{(x,G)}$ is the number of configurations of $X$ occurring at $G$
which extend $x_{\pGcom}$ and project to $\s(x_G)$.

Now for a positive integer $t$ let $R^{(t)}=[-(m+t),m+t-1]$, and
let $(R^{(t)}_n)_{n\in\Z}$ be the partition of $\Z$ by translates of $R^{(t)}$
where $R^{(t)}_0=R^{(t)}$, i.e. $ R^{(t)}_n=[(2n-1)(m+t),(2n+1)(m+t)-1]$. Let
$R^{(0)}=R$. Since
$\left(\sigma\left(Y_{R^{(t)}}\right)\right)_{t=0}^\infty$ and $\left(\sigma\left(X_{R^{(t)}\cap G^c}\right)\right)_{t=0}^\infty$ are increasing sequences with $\bigvee_{t=0}^{\infty} \sigma(Y_{R^{(t)}})=\Balg_Y$ and
$\bigvee_{t=0}^{\infty}\sigma(X_{R^{(t)}\cap G^c})=\sigma(X_{G^c})$ it follows from Lemma~\ref{lem:limit} that
\begin{equation*}
\E\big(1_{\alpha}|\pi^{-1}(\Balg_Y)\vee \sigma(X_{ G^c})\big)(x)=
\begin{cases}
\frac{1}{L_{(x,G)}}
&\alpha\textrm{ extends }x_{\partial G^c},\,
x_G\in\piinv_\text{b}(\s(\alpha))
\\
0&\textrm{otherwise.}
\end{cases}
\end{equation*}
\end{proof}

\section{Class Degree of a factor code}
When $\pi$ is a finite-to-one factor code from a SFT $X$ to a sofic shift
$Y$, there is a uniform upper bound on the number of
pre-images of points in $Y$~\cite{lm}. The minimal number of $\pi$-pre-images of points in $Y$ is called the \textsf{degree} of the code and denoted by $d_\pi$. 
\begin{thm}\label{thm:samenumber}~\cite{lm}
Let $\pi:X\to Y$ be a finite-to-one factor code from a SFT $X$ to an
irreducible sofic shift $Y$. Then every doubly transitive point of $Y$ has exactly $d_{\pi}$
preimages. 
\end{thm}
Theorem 9.1.11 in~\cite{lm} shows this result when $X$ is irreducible. However, the proof only depends on the irreducibility of $Y$.
\begin{cor}~\label{cor:conjdeg}
Let $(X,Y,\pi)$ and $(\wt X,\wt Y,\wt\pi)$ be conjugate factor triples. Then we have $d_\pi=d_{\wt\pi}$.
\end{cor}
\begin{defn}
A family $x^{(1)},\dots,x^{(k)}$ of sequences in a SFT $X$ is \textsf{mutually separated} if, for each integer $i$, $x^{(1)}_i,\dots, x^{(k)}_i$ are all distinct.
\end{defn}
\begin{prop}\label{prop:degreemagic}~\cite{lm}
Let $\pi:X\ra Y$ be a finite-to-one 1-block factor code from a SFT to an irreducible sofic shift $Y$. The preimages of a transitive point are mutually separated. If $\pi$ has a magic symbol $w$ then
$d_\pi=|\piinv_\text{b}(w)|$. 
\end{prop}
For a proof, the reader is referred to Proposition 9.1.9 and Exercise 9.1.3 of~\cite{lm}.

We will find a quantity analogous to the degree when $\pi$ is an infinite-to-one factor
code. This will be done by developing the following equivalence relation
on $X$. Figure~\ref{fig:transition} illustrates Definition~\ref{defn:equiv}.
\begin{defn}\label{defn:equiv}
Suppose $(X,Y,\pi)$ is a factor triple and $x,\,x'\in X$. We say there is a \textsf{transition} from $x$ to
$x'$ and denote it by $x\to x'$ if for each integer $n$, there exists a point $v$ in $X$ so that
\begin{enumerate}
\item $\pi(v)=\pi(x)=\pi(x')$, and
\item $v_{-\infty}^n=x_{-\infty}^n,~~ v^{\infty}_i=x'^{\infty}_i$ for some $i\geq n$.
\end{enumerate}
Write $x\nrightarrow x'$ if the above conditions do not hold. We write $x\sim x'$, and say $x$
and $x'$ are in the same \textsf{transition class} if $x\ra x'$ and $x'\ra x$. The relation $\sim$ is an equivalence relation.
Denote the set of transition classes in $X$ over $y\in Y$ by
$\mathscr{C}_\pi(y)$. Sometimes we denote $\C_\pi(y)$ simply by
$\C(y)$ when there is no ambiguity in understanding $\pi$. Say $[x]\ra
[x']$ if $x\ra x'$ (well-defined). Use the notation $[x]\nrightarrow [x']$ otherwise.
 \end{defn}

\begin{figure}[h]
\centering{
\psfrag{$x$}{$x$}
\psfrag{$x'$}{$x'$}
\psfrag{$y$}{$y$}
\psfrag{$n$}{$n$}
\psfrag{$i$}{$i$}
\psfrag{$v$}{$v$}
\psfrag{p}{$\pi$}
\includegraphics[height=1in]{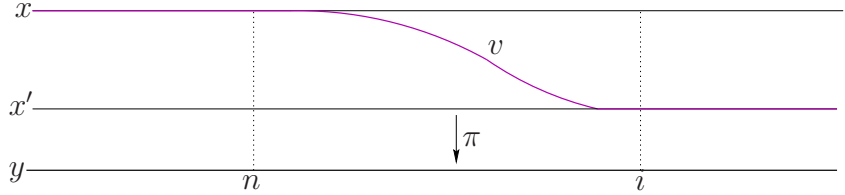}}
\caption{Transition from point $x$ to $x'$}
\label{fig:transition}
\end{figure}

The following fact is derived from Definition~\ref{defn:equiv} immediately.
\begin{fact}\label{fact:class}
Let $\pi:X\ra Y$ be a 1-block factor code from a 1-step SFT
$X$ to a sofic shift $Y$. Let $x,x'$ project to the same point under $\pi$ and $x_{a_i}=x'_{a_i}$ where
$(a_i)_{i\in\mathbb N}$ is a strictly increasing sequence in $\Z$. Then
we have $x\sim x'$.
\end{fact}
Note that Fact~\ref{fact:class} gives an obvious case
when two points lie in the same transition class. Simple examples allow one to find equivalent points without having a common symbol at the same time.

\begin{defn}
Let $(X,Y,\pi)$ be a factor triple. The minimal number of transition classes over points of $Y$ is called the \textsf{class degree} of $\pi$ and denoted by $c_\pi$.
\end{defn}
The following proposition shows that conjugate factor
triples have the same class degree.

\begin{prop}\label{prop:conjinv}
Let  $(X,Y,\pi)$ and $(\wt X,\wt Y,\wt\pi)$ be conjugate factor triples under conjugacies $\phi:X\ra\wt X$ and $\psi:Y\to\wt Y$. Then

 \begin{enumerate}[(a)]
  \item For each $y$ in $Y$, $|\C_\pi(y)|=|\C_{\wt\pi}(\psi(y))|.$
\item  $c_{\pi}=c_{\wt\pi}.$ 
 \end{enumerate}

\end{prop}

\begin{proof}
(a) is clear by noting that given $u,v$ in $X$, $\phi(u)\sim\phi(v)$ if and only if $u\sim v$, and (b) is a direct corollary of (a).    
\end{proof}

The following two theorems give an upper bound on the number of transition classes over points of $Y$. Later in Theorem~\ref{thm:classdegree} we give a finitary characterization of the number of such classes.

\begin{thm}\label{thm:symboloccurs}
Let $\pi:X\ra Y$ be a $1$-block factor code from a 1-step SFT $X$ to a
sofic shift $Y$. Let $y$ in $Y$ and let $w$ be a symbol in $Y$ occurring infinitely often in positive coordinates of $y$. Then $|\C(y)|\leq|\piinv_\text{b}(w)|$.
\end{thm}
\begin{proof}
Let $(a_j)_{j\in\mathbb N}$ be a strictly increasing sequence of integers such that $y_{a_j}=w$ for each $j\in\mathbb N$. Let $|\piinv_\text{b}(w)|=d$ and suppose $\pi^{-1}(y)$ contains $d+1$
distinct transition classes $C_1,\dots,C_{d+1}$. Form a set
$A=\{x_1,\dots,x_{d+1}\}$ where $x_i$ is an arbitrary point of $C_i$.
Since $\piinv_\text{b}(w)$ contains exactly $d$ symbols, it follows that for
each $j\in\mathbb N$ there are at least two points in $A$ with the same
$a_j$th coordinate. The Pigeonhole Principle implies that there is a
subsequence $(b_{k})_{k\in\mathbb N}$ of $(a_j)_{j\in\mathbb N}$ and at least two points $x$ and
$x'$ in $A$ with $x_{b_k}=x'_{b_k}$ for each $k\in\mathbb N$. Fact~\ref{fact:class} implies that
$x\sim x'$. This contradicts the assumption that $x$ and $x'$
are in different transition classes.
\end{proof}
A similar proof to the above implies the following theorem. 
\begin{thm}\label{thm:classsize}
Let $\pi:X\ra Y$ be a $1$-block factor code from a 1-step SFT $X$ to a
sofic shift $Y$. Then $|\C(y)|<\infty$ for each $y\in Y$. Moreover, \DS{min_{w\in\A(Y)}\{|\piinv_\text{b}(w)|\}} is the uniform upper bound on the number transition classes over each right transitive point of $Y$.
\end{thm}

We show when
$\pi:X\ra Y$ is a factor code from a SFT $X$ to an irreducible sofic shift $Y$, there are exactly $c_\pi$ transition classes over a transitive point of $Y$. In order to show this, we introduce another quantity $c^*_\pi$ in Definition~\ref{defn:cStar}, defined concretely in
terms of blocks. Since the class degree is invariant under conjugacy of factor triples, by recoding we may focus only on 1-step SFTs and 1-block factor codes. 

\begin{defn}\label{defn:TB}
Let $\pi:X\ra Y$ be a $1$-block factor code from a 1-step SFT $X$ to a
sofic shift $Y$. Let $W$ be a $Y$-block of length $p+1$. Let $n$ be an integer in $(0,p)$, and $M$ be a subset of $\piinv_\text{b}(W_n)$. We say $U\in\piinv_\text{b}(W)$ \textsf{is routable through} $a\in M$ \textsf{at time} $n$ if there is a block $U'\in\piinv_\text{b}(W)$ with
$U'_0=U_0,~U'_{n}=a$, and $U'_{p}=U_{p}$. A triple $(W,n,M)$ is called a
\textsf{transition block} of $\pi$ if every block in $\piinv_\text{b}(W)$ is routable through a symbol of $M$ at time $n$. The cardinality of the set $M$ is called the \textsf{depth} of the
transition block $(W,n,M)$.
\end{defn}
Figure~\ref{fig:depth}
illustrates Definition~\ref{defn:TB}. 
\begin{figure}[h!]
\centering
\psfrag{$U$}{$U$}
\psfrag{$V$}{$V$}
\psfrag{$K$}{$K$}
\psfrag{$U'$}{$U'$}
\psfrag{$V'$}{$V'$}
\psfrag{$K'$}{$K'$}
\psfrag{$a_1$}{$a_1$}
\psfrag{$a_2$}{$a_2$}
\psfrag{$W_0$}{$W_0$}
\psfrag{$W_{p}$}{$W_p$}
\psfrag{$p$}{$\pi$}
\psfrag{$n$}{$n$}
\includegraphics[height=2in]{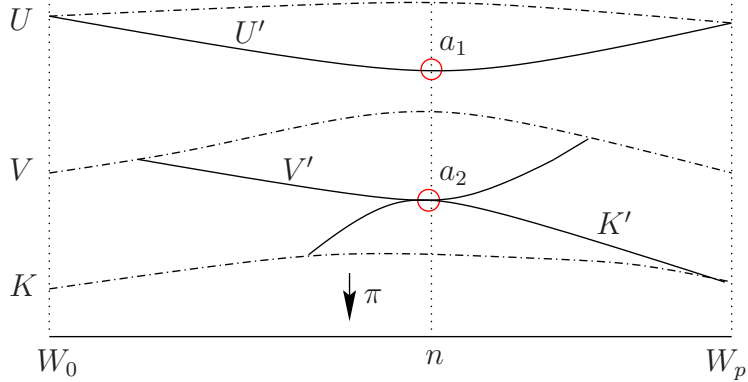}
\caption{$(W,n,M)$ is a transition block with $M=\{a_1,a_2\}$. The blocks $U,\,V,\,K\in\piinv_\text{b}(W)$ are routable through members of $M$ at time $n$ via blocks $U',\,V',\,K'\in\piinv_\text{b}(W)$.}
\label{fig:depth}
\end{figure}
\begin{defn}\label{defn:cStar}
Let
$$c^*_{\pi}=\min\{|M|\colon (W,n,M)\textrm{ is a transition block of $\pi$}\}.$$
A \textsf{minimal transition block} of $\pi$ is a
transition block of depth $c^*_\pi$.
\end{defn}
\begin{ex}\label{ex:MTB}
In figure~\ref{fig:MTB}, we display an example of a labeled graph which
defines an infinite-to-one 1-block factor code $\pi$. We see that
$(001,\,1,\,\{b\})$ is a minimal transition block of $\pi$ of depth 1. For example, observe that block $U=aac$ is routable through $b$ at time 1 by considering $U'=abc$.
\begin{figure}[H]
\centering
\psfrag{$a$}{$a$}
\psfrag{$b$}{$b$}
\psfrag{$c$}{$c$}
\psfrag{$d$}{$d$}
\psfrag{0}{$0$}
\psfrag{1}{$1$}
\includegraphics[height=1.2in]{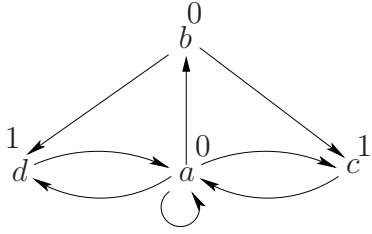}
\caption{Graph for Example~\ref{ex:MTB}}
\label{fig:MTB}
\end{figure}

\end{ex}
We need to develop some definitions and lemmas before proving Theorem~\ref{thm:classdegree} below. Given a factor triple $(X,Y,\pi)$, a point $y$ in $Y$ and a block $V$ occurring in $y$, it may be the case that some preimages of $V$ can not be extended to a point in $X$ which projects to $y$ under $\pi$. Here we show that in order to recognize those non-extendable blocks we do not need to deal with a bi-infinite sequence. A finite extension of a block is sufficient to determine if the block can be extended to a point above $y$.

\begin{defn}\label{defn:exten}
Let $(X,Y,\pi)$ be a factor triple and $y$ be in $Y$. Let $V$ be a block of $Y$ occurring in $y$ at a $\Z$-interval $[m,n]$. A Block $W$ of $Y$ is called a \textsf{$y$-synchronizing extension} of $V$ if $W$ occurs in $y$ at $[m-l,n+l]$ for some integer $l\geq 0$ and
\begin{equation*}
\begin{split}
\{U\in\piinv_\text{b}(V)\colon& \exists x\in\piinv(y),\,x_{[m,n]}=U\}
\\
&=\{U\in\piinv_\text{b}(V)\colon \exists\, B\in\piinv_\text{b}(W),\,B_{[l,l+n-m]}=U\}.
\end{split}
\end{equation*}
\end{defn}

\begin{lem}\label{lem:exten}
Let $(X,Y,\pi)$ be a factor triple and $y$ in $Y$. Let $V$ be a block occurring in $y$ at a $\Z$-interval $[m,n]$. There is a $y$-synchronizing extension of $V$.
\end{lem}

\begin{proof}
Let $$S=\{U\in\piinv_\text{b}(V)\colon \exists\, x\in\piinv(y),\,x_{[m,n]}=U\}.$$
For $l\geq 0$ let 
$$S^l=\{U\in\piinv_\text{b}(V)\colon \exists\, B^l\in\piinv_\text{b}(y_{[m-l,n+l]}),\, B^l_{[l,l+n-m]}=U\}.$$

We show for some $l$ we have $S=S^l$. Then it will imply that $y_{[m-l,n+l]}$ is a $y$-synchronizing extension of $V$.

Clearly for all $l\geq 0$ we have $S\subseteq S^l$. We show $S^l\subseteq S$ for some $l$ as follows. First note that since $\pi$ is a factor code, $S$ is not empty. It follows that the sequence $(S^l)_{l=0}^\infty$ is a non-increasing sequence of non-empty finite sets. Thus there is $l$ such that for each $r\geq l$ we have $S^l=S^{r}$. We show $S^l\subseteq S$. 

Let $U$ be in $S^l$. Then $U$ is in $S^r$ for each $r\geq l$. It follows that for each $r\geq l$ there is $B^r$ in $\piinv_\text{b}(y_{[m-r,n+r]})$ with $B^r_{[r,r+n-m]}=U$. For each $r\geq l$ consider the cylinder set $C_{m-r}(B_r)$. The sequence $(C_{m-r}(B_r))_{r\geq l}$ is a non-increasing sequence of non-empty compact subsets of $X$. It follows that $\bigcap_{r\geq l}C_{m-r}(B_r)$ is not empty. Let $x\in \bigcap_{r\geq l}C_{m-r}(B_r)$. Then $x_{[m,n]}=U$ and $\pi(x)=y$ which implies that $U$ is in $S$.
\end{proof}
Here we introduce a relation between symbols and transition classes.
\begin{defn}\label{defn:belong}
Let $\pi:X\ra Y$ be a $1$-block factor code from a 1-step SFT $X$ to a sofic
shift $Y$. Let $y$ be in $Y$. Let $n$ be in $\Z$, $i$ be a symbol in $\piinv_\text{b}(y_n)$, and $C$ be a transition class in $\C(y)$. Say $i$ \textsf{belongs to} $C$ \textsf{at time} $n$, and denote it by $i\in S_n(C)$, if
$$\{D\in\C(y)\colon
\textrm{there is } x \textrm{ in } D \textrm{ with }x_n=i\}=\{D\in\C(y):C\ra D\}.$$ In other words, the set of transition classes in $\C(y)$ containing a point with $n$th
coordinate being $i$, is the same as the set of transition classes to which there is a transition from the class $C$ (including $C$ itself).

Say $i$ is
\textsf{transient at time $n$} if there is no transition class $C$ in $\C(y)$ for which $i\in
S_n(C)$.
\end{defn}
Example~\ref{ex:transient} illustrates the definition above.
\begin{ex}\label{ex:transient}
Consider the directed labeled graph in Figure~\ref{fig:transient} which
presents a 1-block factor code $\pi:X\to Y$ where
$X\subseteq\{a,\,b,\,c,\,d,\,e,\,f,\,g\}^\Z$ and $Y\subseteq\{0,\,1\}^\Z$. Let $y$ be the point $\cdots 01\overset{*}0 1010\cdots$ in $Y$. By
Definition~\ref{defn:equiv}, there are 3 distinct transition classes in $\C(y)$ as
follows:

\begin{figure}[h]
\centering
\psfrag{a}{$a$}
\psfrag{b}{$b$}
\psfrag{c}{$c$}
\psfrag{d}{$d$}
\psfrag{e}{$e$}
\psfrag{f}{$f$}
\psfrag{g}{$g$}
\psfrag{0}{$0$}
\psfrag{1}{$1$}
\includegraphics[height=1in]{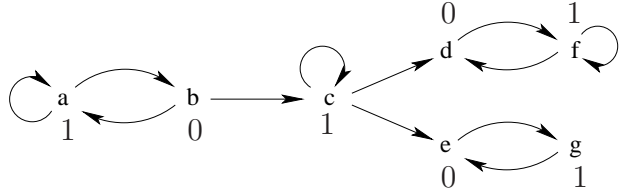}
\caption{Graph for Example~\ref{ex:transient}}
\label{fig:transient}
\end{figure}
\begin{enumerate}
\item Transition Class $C_1=[x_1]=\{x_1\}$ where $x_1=\cdots ba\overset{*}babab\cdots$.
\item
Transition Class $C_2=[x_2]$ where $x_2$ is a point in $\piinv(y)$ which does not contain symbols $e$ or $g$ but contains only $d$'s
and $f$'s from some time onwards, for example: $x_2=\cdots ba\overset{*}bcdfd\cdots$.
\item
Transition Class $C_3=[x_3]$ where $x_3$ is a point in $\piinv(y)$ which does not contain symbols $d$ or $f$ but contains only $e$'s
and $g$'s from some time onwards, for example: $x_3=\cdots eg\overset{*}egege\cdots$.
\end{enumerate}
Clearly $C_1\ra C_2$ and $C_1\ra C_3$, but not vice versa. $C_2\nrightarrow C_3$, $C_3\nrightarrow C_2$, and symbol $c$ is transient at any time.

Note that the class degree of the factor code $\pi$ is 2 (consider a point of $Y$ in which the block $11$ occurs at some time).
\end{ex}

\begin{lem}\label{nontransient}
Let $\pi:X\ra Y$ be a $1$-block factor code from a 1-step SFT $X$ to a
sofic shift $Y$, and $x$ be a point in $X$. There is 
$m<\infty$ such that for each $n\geq m$ the symbol $x_n$ is not transient.
In fact, $m$ can be found in such a way that for each $n\geq m$, $x_n$
belongs to the class $[x]$ at time $n$; i.e., $x_n\in S_n[x]$.
\end{lem}

\begin{proof}
Let $\pi(x)=y$. Suppose that the transition class $C$ in $\C(y)$ satisfies $[x]\nrightarrow C$. Then there exists $i<\infty$ such that for any $z$ in $\piinv(y)$
with $z_i=x_i$ we have $z\notin C$. Denote the smallest such $i$ by $i^C$. Note that $i^C$ may be $-\infty$ (but not $+\infty$). Let $m=\max\{i^C: C\in\C(y),\, [x]\nrightarrow C\}$. Let $n\geq m$. We show \begin{equation}\label{eq:2sets}
\{C\in\C(y):[x]\ra C\}=\{C\in\C(y)\colon
\textrm{there is } z\in C \textrm{ with }z_n=x_n\},
\end{equation}
which will imply the result $x_n\in S_n[x].$

Let  $C$ be a transition class in $\C(y)$ with $[x]\to C$. It follows that $x\to z$ for some $z$ in $C$. Therefore, there is a point $u$ in $\piinv(y)$ as follows.
\begin{equation*}
u_t=
\begin{cases}
{x_t}& \textrm{ if }-\infty<t\leq n
\\
z_t&\textrm{ if }l\leq t<\infty,
\end{cases}
\end{equation*}
for some $l\geq n$. Since $u_t=z_t$ from time $l$ onwards, it follows that $u$ is in the class $C$. Moreover, we have $u_n=x_n$ by the construction. It follows that $C$ belongs to the set in the right hand side of Equation~\eqref{eq:2sets}.

Now let $C$ be an equivalence class in $\C(y)$ with $[x]\nrightarrow C$. Suppose, for a contradiction, $C$ contains a point $z$ with $z_n=x_n$. Form the point $u$ as follows.
\begin{equation*}
u_t=
\begin{cases}
{x_t}& \textrm{ if }-\infty<t\leq n
\\
z_t&\textrm{ if }n< t<\infty.
\end{cases}
\end{equation*}
Note that since $X$ is a 1-step SFT and $z_n=x_n$, the point $u$ is in $X$. Moreover $u$ maps to $y$ by the construction. Noticing that $u_t=z_t$ from the time $n$ onwards, shows that $u$ lies in the class $C$. Recall that $n\geq i^C$. By the argument in the first paragraph of the proof, any point in $\piinv(y)$ with $i^C$th coordinate equal to $x_{i^C}$ does not belong to $C$. This contradicts that $u$ is a point in $C$.  
\end{proof}
\begin{lem}\label{lem:S(C)}
Let $\pi:X\ra Y$ be a $1$-block factor code from a 1-step SFT $X$ to
a sofic shift $Y$. Let $y\in Y$. There is $m<\infty$ such that for
each $n\geq m$ and $C$ in $\C(y)$ there is a symbol $i$ in $\sinv(y_{n})$ which belongs to the class $C$ at time $n$; i.e., $i\in S_{n}(C)$.
\end{lem}

\begin{proof}
Let $\C(y)=\{C_1,\dots,C_d\}$ for some finite $d$. Let
$A=\{x^{(1)},\dots,x^{(d)}\}$ be a set containing an arbitrary point
$x^{(i)}$ of $C_i$ for each $C_i$. By Lemma~\ref{nontransient} there is a $n_{C_i}<\infty$ such that for each
$n\geq n_{C_i}$ the symbol $x_n^{(i)}$ belongs to $C_i$ at time $n$. Let
$m=\sup\{n_{C_i}\colon 1\leq i\leq d\}$. Note that $m<\infty$. Moreover, for each
$n\geq m$ and $C_i$ the symbol $x^{(i)}_n$ belongs to the class $C_i$ at time $n$.
\end{proof}
\begin{prop}\label{prop:Ypi}
Let $\pi:X\ra Y$ be a $1$-block factor code from a 1-step SFT $X$ to
a sofic shift $Y$. Let $\nu$ be an invariant measure on $Y$. Then
$$
\nu\left(\{y\in Y\colon \forall C\in\C(y),\forall n\in\Z\text{ there is }
i\in\piinv_\text{b}(y_n)\text{ with } i\in S_n(C)\}\right)=1.
$$
\end{prop}
\begin{proof}
Let $y$ in $Y$ and let $m(y)$ be the infimum of the set of $m$'s with the properties given in the statement of Lemma~\ref{lem:S(C)}. We show $$\nu\left(\{y\in Y\colon m(y)=-\infty\}\right)=1$$ which will imply the result immediately. Note that $C\in\C(y)$ if and only if $T(C)\in\C(T(y))$. This shows $m(T(y))=m(y)-1$. For $k<\infty$ let $A_k=\{y\in Y: m(y)=k\}$. We have $T(A_k)=A_{k-1}$. Since $\nu$ is $T$-invariant it follows that $\nu(A_k)=0$. Therefore $m(y)=-\infty$ for $\nu$-a.e. $y$ in $Y$.
\end{proof}

\begin{defn}
Let $X$ be a shift space. A point $x$ in $X$ is \textsf{recurrent} if for each block $B$ which occurs in $x$, $B$ occurs infinitely often in the positive coordinates of $x$.
\end{defn}

Note that given any invariant measure $\mu$ on $X$, $\mu$-almost every point of $X$ is recurrent.

\begin{thm}\label{thm:classdegree}
Let $\pi$ be a 1-block factor code from a 1-step SFT $X$ to a sofic shift $Y$. Let $y$ be a point of $Y$ and $d$ be the minimal depth of a transition block occurring in $y$. Then  $d\leq|\C(y)|$. Moreover, if $y$ is recurrent then $d=|\C(y)|$.
\end{thm}

\begin{proof}
We show $d\leq|\C(y)|$ by finding a transition block of depth $|\C(y)|$ occurring in $y$. We construct such a transition block in the following 4 stages. Let $h=|\C(y)|$ and suppose
$\C(y)=\{C_1,\dots,C_h\}$. Choose an integer
$n_1$ satisfying the properties given in the statement of Lemma~\ref{lem:S(C)}.

\noindent\textbf {Stage {$\emph 1.$}} We claim there is $n_2\in [n_1,\infty)$ such that
for each $x\in\pi^{-1}(y)$, there is $n_1\leq t\leq
n_2$ so that $x_t$ is not transient.

\begin{proof} Suppose there is no such $n_2$. It follows that for each $j\geq n_1$ there is a point $x^{(j)}$ in $\piinv(y)$ such that $x_l^{(j)}$ is transient for all $n_1\leq l\leq j$. Consider the sequence $(x^{(j)})_{j\in\Z}$, and let $x$ be the limit of a convergent subsequence of it. Clearly $x\in\piinv(y)$. However, $x_l$ is transient for each $l\geq n_1$, contradicting Lemma~\ref{nontransient}.
\end{proof}

\noindent\textbf {Stage {$\emph 2.$}} We claim there is $n_3\in [n_2,\infty)$ and a set of symbols $$M'=\{a_1,\dots,a_h\}$$ with
$a_e\in S_{n_3}(C_e)$ such that for each $i$ in $S_{n_2}(C_e)$ there is a
block $$U\in\sinv(y_{n_2
}y_{{n_2}+1}\dots y_{{n_3}-1}y_{n_3})$$ which begins with $i$ and ends with $a_e$. See Figure~\ref{fig:n2}. 
\begin{figure}[h]
\centering{
\psfrag{$C_1$}{$C_1$}
\psfrag{$C_2$}{$C_2$}
\psfrag{$C_3$}{$C_3$}
\psfrag{$i_1$}{$i_1$}
\psfrag{$i_2$}{$i_2$}
\psfrag{$i_3$}{$i_3$}
\psfrag{$i_4$}{$i_4$}
\psfrag{$i_5$}{$i_5$}
\psfrag{$i_6$}{$i_6$}
\psfrag{$a_1$}{$a_1$}
\psfrag{$a_2$}{$a_2$}
\psfrag{$a_3$}{$a_3$}
\psfrag{$n_0$}{$n_1$}
\psfrag{$n_1$}{$n_2$}
\psfrag{$n_2$}{$n_3$}
\psfrag{$y$}{$y$}
\psfrag{p}{$\pi$}
\includegraphics{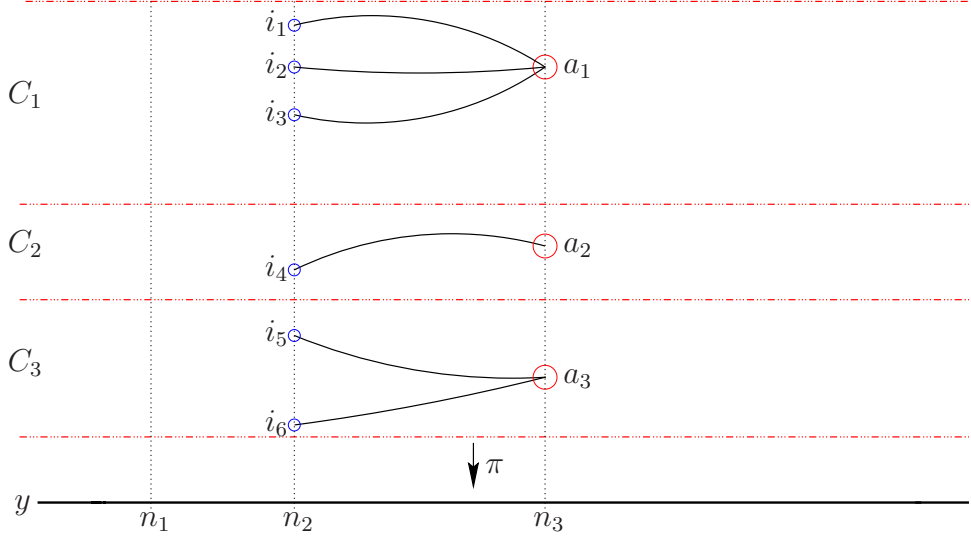}
\caption{An example illustrating Stage 2. $\C(y)=\{C_1,~C_2,~C_3\}$, $i_1,\,i_2,\,i_3\in S_{n_2}(C_1)$, $i_4\in S_{n_2}(C_2)$, $i_5,\,i_6\in S_{n_2}(C_3)$, and $M'=\{a_1,~a_2,~a_3\}$.}
\label{fig:n2}}
\end{figure}
\begin{proof}Let $C_e$ be a transition class in $\C(y)$ and $x^e$ be a point in $C_e$. For each symbol $i$ in $S_{n_2}(C_e)$ (non-empty by the choice of $n_1$), there is a bi-infinite sequence
$z^i$ in the class $C_e$ with $z^i_{n_2}=i$ such that $z^i$ matches $x^e$ from some
time $k^i\in [n_2,\infty)$ onwards. Let $$k^e=\max\{k^i: i\in
S_{n_2}(C_e)\},$$ and $$n_3=\max\{k^e: C_e\in\C(y)\}.$$ Note that $n_3<\infty$. Rename
$x_{n_3}^e$ for each $C_e\in\C(y)$ as $a_e$, and let $M'=\{a_{1},\dots,
a_{h}\}$.
\end{proof}

\noindent\textbf {Stage {$\emph 3.$}} We claim there is $n_4\in [n_3,\infty)$ such that for each point $x$ in $\pi^{-1}(y)$ there is a point $x'$ in $\pi^{-1}(y)$ so that $x'_r=x_r$ for every $r\in (-\infty,n_1]\cup [n_4,\infty)$, and $x'_{n_3}\in M'$. See
Figure~\ref{fig:n3}.

\begin{figure}[h]
\centering{
\psfrag{$C_1$}{$C_1$}
\psfrag{$C_2$}{$C_2$}
\psfrag{$C_3$}{$C_3$}
\psfrag{$x$}{$x$}
\psfrag{$z$}{$z$}
\psfrag{$u$}{$u$}
\psfrag{$x'$}{$x'$}
\psfrag{$z'$}{$z'$}
\psfrag{$u'$}{$u'$}
\psfrag{$a_1$}{$a_1$}
\psfrag{$a_2$}{$a_2$}
\psfrag{$a_3$}{$a_3$}
\psfrag{$n_0$}{$n_1$}
\psfrag{$n_1$}{$n_2$}
\psfrag{$n_2$}{$n_3$}
\psfrag{$n_3$}{$n_4$}
\psfrag{$y$}{$y$}
\psfrag{p}{$\pi$}
\includegraphics{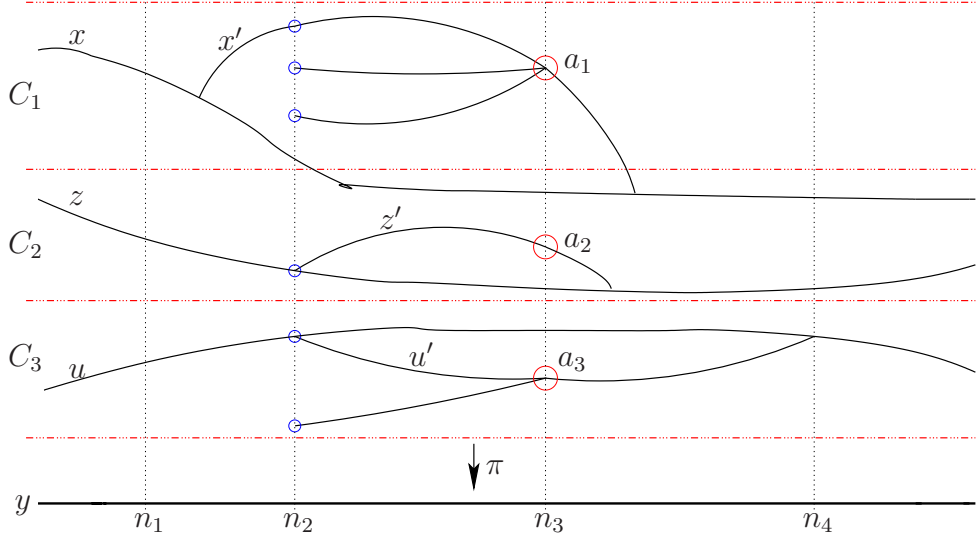}
\caption{Graph for Stage 3. $M'=\{a_1,\,a_2,\,a_3\}$. $x,x'\in\piinv(y)$, $x'_r=x_r$ for each $r\in (-\infty,n_1]\cup [n_4,\infty)$, and $x'_{n_3}\in M'$. Same for $z,\,z'$ and $u,\,u'$.}
\label{fig:n3}}
\end{figure}
\begin{proof}
Let $x$ be a point in $\piinv(y)$. By stage 1, there is $n_1\leq t\leq n_2$ such
that $x_t$ is not transient; i.e., $x_t$ belongs to some transition class $C_e$ at time $t$. It follows that there is a point $u$ in the class $C_e$ with $u_t=x_t$. The new block 
\begin{equation}\label{eq:firstblock}
\dots x_{n_1}\dots x_tu_{t+1}\dots u_{n_2}
\end{equation}
 is an allowable block of $X$. Moreover, this block maps to the block $$\dots y_t\dots
y_{n_2}.$$ Note that the symbol $u_{n_2}$ belongs to the class $C_e$ at time $n_2$. So by Stage 2, there is a block starting at
$u_{n_2}$ ending at $a_e$ which maps to $$y_{n_2}\dots y_{n_3}.$$ Denote this block by \begin{equation}\label{eq:secblock}
u_{n_2}v_{n_2+1}\dots
v_{n_3-1}a_e.
\end{equation}
Connect blocks in Equations~\eqref{eq:firstblock} and~\eqref{eq:secblock} at $u_{n_2}$ to get the following allowable block of $X$
\begin{equation*}
\dots x_{n_1}\dots
x_tu_{t+1}\dots u_{n_2}v_{n_2+1}\dots v_{n_3-1}a_e
\end{equation*}
Observe that
having $x_t$ in $S_t(C_e)$ implies that $x$ must belong to a transition class
$C_f$ with $C_e\ra C_f$. On the other hand having $a_e$ in $S_{n_3}(C_e)$ implies that there is a point $b$ in $C_e$ with
$$b_{(-\infty,\,n_3]}=\dots x_{n_1}\dots x_tu_{t+1}\dots
u_{n_2}v_{n_2+1}\dots v_{n_3-1}a_e$$ which matches $x$ from some time
$j$ onwards for some $n_3\leq j <\infty$. Let
$$n_x=\min_{n_3\leq j<\infty}\{j: \exists x'\in \pi^{-1}(y) \textrm{ with }
x'_r=x_r \forall r\in (-\infty,n_1]\cup [j,\infty),\,~x'_{n_3}=a_e\}.$$
We claim that there
is $n^e<\infty$ such that for each $x\in \pi^{-1}(y)$ with $x_t\in
S_t(C_e)$ for some $n_1\leq t\leq n_2$, there exists a point $x'\in
\pi^{-1}(y)$ with $x'_r=x_r$ for each $r\in (\infty,n_1]\cup [n^e,\infty)$, and $x'_{n_3}=a_e$. Then
letting $$n_4=\max\{n^e: C_e\in\C(y)\}<\infty$$ will complete the proof of Stage 3.

Suppose, for a contradiction, that there does not exist such an $n^e$. Then there is
a sequence $x^{(l)}$ with $x^{(l)}_{t_l}\in S_{t_l}(C_e)$ for some $n_1\leq
t_l\leq n_2$, such that $\lim_{l\to\infty} n_{x^{(l)}}=\infty$. Let $x^*$
be the limit of a convergent subsequence of $x^{(l)}$. Clearly $x^*$ is in 
$\pi^{-1}(y)$ and $n_{x^*}=\infty$ which contradicts the fact that for each $x$ in $\piinv(y)$, $n_x<\infty$. 
\end{proof}
{\bf Stage {\emph 4.}} Let $W$ be a $y$-synchronizing extension of the block $$V=y_{n_1}\dots y_{n_4}$$ occurring in $[n_0,n_5]$ for some $n_0\leq n_1$ and $n_5\geq n_4$. We claim $(W,n_3-n_0,M')$ is a transition block
of depth $|\C(y)|$.
\begin{proof}
Let $U\in\sinv(W)$. By the definition of $y$-synchronizing extension, there is $x$ in $\piinv(y)$ with $x_{[n_1,\,n_4]}=U$. By Stages 1, 2, and 3 there is a point
$x'$ in $\piinv(y)$ with $x'_r=x_r$ for each $r\in (-\infty,n_1]\cup[n_4,\infty)$, and $x'_{n_3}\in M'$. Having
$x'_{n_1}=x_{n_1}$, $x'_{n_3}\in M'$,  $x'_{n_4}=x_{n_4}$ and $\s(x'_{[n_1,\,n_4]})=V$ simply means that
$U$ is routable through a symbol of $M'$ at time $n_3-n_0$. Since
$U$ in $\sinv(W)$ is arbitrary it follows that $(W,\,n_3-n_0,\,M')$ is a transition block. Because $M'$ has $|\C(y)|$ elements it follows that $W$ is a transition block of depth $|\C(y)|$.
\end{proof}
Now suppose $y$ is recurrent. Then a transition block of depth $d$ which occurs in $y$, it occurs infinitely often in the positive coordinates of $y$. Apply the same method which was used in the proof of Theorem~\ref{thm:symboloccurs} to see $d\geq |\C(y)|$. Since by the first part of the theorem $d\leq |\C(y)|$ we obtain $d=|\C(y)|$ when $y$ is recurrent.
\end{proof}
\begin{cor}\label{cor:classdegree}
Let $\pi:X\ra Y$ be a factor code from a SFT $X$ to an
irreducible sofic shift $Y$. Then
\begin{enumerate}[(a)]
\item $c_\pi=c^*_\pi$.
\item There are exactly $c_{\pi}$ transition classes over every right transitive point of $Y$.
\item Given an ergodic measure $\nu$ on $Y$, $\nu$-almost every point of $y$ has the same number of transition classes over it. We call this number $|\C(\nu)|$.
\end{enumerate}
\end{cor}
\begin{proof}
 Let $(W,n,M)$ be a transition block of $\pi$ of depth $c^*_\pi$ and let $y$ be a right transitive point of $Y$. Notice that $W$ occurs infinitely often in the positive coordinates of $y$. Then (a) and (b) are clear by Theorem~\ref{thm:classdegree}.

Now let $\nu$ be an ergodic measure on $Y$. Let $(W',n',M')$ be a transition block of $\pi$ with $\nu[W']>0$ such that $$|M'|=\min\{|\wt M|\colon (\wt W,\wt n,\wt M)\text{ is a transition block of }\pi,\,\nu[\wt W]>0\}.$$ Note that $W'$ occurs infinitely often in the positive coordinates of $\nu$-almost every point of $Y$. Thus (c) is obtained as a direct corollary of Theorem~\ref{thm:classdegree}. 
\end{proof}

We now show that for a
finite-to-one factor code, the degree and the class degree of
the code are the same.

\begin{thm}\label{thm:finitemagic}
Let $\pi:X\ra Y$ be a finite-to-one factor code from a SFT to an
irreducible sofic shift $Y$. Then $c_\pi=d_\pi$.
\end{thm}
\begin{proof}
Corollary~\ref{cor:conjdeg} and Proposition~\ref{prop:conjinv}(b) imply that the degree of a code and the class degree of a code are both
invariant under conjugacy of factor triples. So by recoding, without loss of generality, we may assume $X$ is a 1-step SFT and $\pi$ is a 1-block factor code. Let $y$ be a transitive point of $Y$. By Proposition~\ref{prop:degreemagic} the preimages of $y$ are mutually separated. It follows that each transition class above $y$ has only one point. Therefore, $c_\pi=|\piinv(y)|=d_\pi$.
\end{proof}

\section{Bounding the number of ergodic measures of relative maximal entropy}
In Section~\ref{sec:unif} we mentioned that although every 1-dimensional
irreducible SFT has a unique ergodic measure (Parry measure) of maximal entropy,
there can be more than one ergodic measure of relative maximal entropy over an ergodic measure; i.e., given a factor triple $(X,Y,\pi)$ and an ergodic measure $\nu$ on $Y$, there can exist more than one ergodic measure on $X$ that projects to $\nu$ under $\pi$ and has maximal entropy among measures in the fibre $\piinv\{\nu\}$, see \cite[Example~3.3]{pqs}. In this
section we show that the number of such measures over a fully supported ergodic measure can be no more than the class degree.

\begin{thm}\label{thm:final}
Let $(X,Y,\pi)$ be a factor triple and $\nu$ be a fully supported ergodic measure on $Y$. The number of ergodic measures of relative maximal entropy over $\nu$ is at most $c_\pi$.
\end{thm}

Since entropy is a conjugacy invariant the following observation follows
immediately.
\begin{obs}\label{obs:numberofMRME}
Let $(X,Y,\pi)$ and $(\wt X,\wt Y,\tilde\pi)$ be conjugate factor triples. Let $\nu$ be an ergodic measure on $Y$ and $\tilde\nu$ be its corresponding ergodic measure on $\wt Y$. The number of ergodic measures of relative maximal entropy over $\nu$ is the same
as the number of ergodic measures of relative maximal entropy over $\tilde\nu$.
\end{obs}
In 2003, Petersen, Quas, and Shin \cite{pqs} found an upper
bound on the number of measures of relative maximal entropy. Let $N_\nu(\pi)$ denote the minimum number of symbols in $\piinv_\text{b}(w)$ as $w$ runs over the symbols in $\A(Y)$ for which $\nu[w]>0$.
\begin{thm}\cite[Corollary 1]{pqs}\label{pqs}
Let $\pi:X\ra Y$ be a 1-block code from a 1-step SFT $X$ to a sofic shift
$Y$. Let $\nu$ be an ergodic measure on $Y$. The number of
ergodic measures of maximal entropy over $\nu$ is at
most $N_\nu(\pi)$.
\end{thm}

This bound suffers from not being invariant under conjugacy. For example,
the full 2-shift and any higher block presentation of it give different
bounds on the number of ergodic measures of maximal entropy which map to
the trivial measure on the full 1-shift.

One possibility to avoid having a non-invariant upper bound is to take
the minimum of the bound in \cite{pqs} over all conjugate factor triples with 1-block factor codes. This even improves the original bound. However, there is no known algorithm for computing this minimum. 

Here we show that given an ergodic measure $\nu$ on $Y$, the number of transition classes over a $\nu$-generic point of $Y$ is an upper bound on the number of measures of relative maximal entropy over $\nu$. Proposition~\ref{prop:conjinv}(b) verifies
that this bound is invariant under conjugacy of factor triples and Proposition~\ref{minbound} shows that it beats the bound mentioned above obtained by minimizing the bound in \cite{pqs} over conjugate factor triples. 

\begin{prop}\label{minbound}
Let $\pi:X\to Y$ be a 1-block factor code from a SFT $X$ to a sofic shift $Y$. Let $\nu$ be an ergodic measure on $Y$. Then we have $|\C_\pi(y)|\leq\min\{N_\nu(\tilde\pi)\colon
(X,Y,\pi)\cong(\tilde X,\wt Y,\tilde\pi),\, \wt\pi \text{ is 1-block}\}$ where $y$ is a $\nu$-generic point of $Y$. Equality holds if $\pi$ is finite-to-one and $\nu$ is fully supported.
\end{prop}
\begin{proof} 
Let the minimum of $N_\nu(\tilde\pi)$ over all factor triples $(\wt X,\wt Y,\tilde\pi)$ which are conjugate to $(X,Y,\pi)$ and $\wt\pi$ is 1-block, be attained at $(\bar X,\bar Y,\bar\pi)$ where $X$ is conjugate to $\bar X$ under a conjugacy $\phi:X\to\bar X$, $Y$ is conjugate to $\bar Y$ under a conjugacy $\psi:Y\to\bar Y$, and $\bar\pi\circ\phi=\psi\circ\pi$. Denote the measure $\nu\circ\psi^{-1}$ by $\bar\nu$. Suppose that the minimum of  $|\bar\pi^{-1}_\text{b}(w)|$ over all $w$ in $\A(\bar Y)$ with $\bar\nu[w]>0$ is attained at a symbol $\bar w$. Let $\bar y$ be a $\bar\nu$-generic point of $\bar Y$. There
is a strictly increasing sequence of integers $(a_i)_{i\in\mathbb N}$ with
$\bar y_{a_i}=\bar w$. Theorem~\ref{thm:classsize} implies that $|\C_{\bar\pi}(\bar y)|\leq
|\bar\pi^{-1}_\text{b}(\bar w)|$. Let $y=\psi^{-1}(\bar y)$. Then by Proposition~\ref{prop:conjinv}(a) we have $|\C_\pi(y)|=|\C_{\bar\pi}(\bar y)|\leq
|\bar\pi^{-1}_\text{b}(\bar w)|$ which completes the proof of the first part of the result.

Now suppose $\pi$ is finite-to-one and $\nu$ is fully supported. Since $\nu$ is fully supported we have 
$$N_\nu(\tilde\pi)=\min\{|\wt\pi^{-1}_\text{b}(w)|\colon w\in\A(\wt Y)\}.$$ 
Therefore the minimum of $N_\nu(\tilde\pi)$ over all $(\wt X,\wt Y,\tilde\pi)$ which are conjugate to $(X,Y,\pi)$ and $\wt\pi$ is 1-block is attained at a factor triple $(\bar X,\bar Y,\bar\pi)$ with a magic symbol $\bar w\in \A(\bar Y)$. By Proposition~\ref{prop:degreemagic} we have $|\bar {\pi}^{-1}_\text{b}(\bar w)|=d_{\bar \pi}$ which implies
$$\min\{N_\nu(\tilde\pi)\colon (X,Y,\pi)\cong(\bar X,\bar Y,\tilde\pi),\, \bar\pi \text{ is 1-block}\}=d_{\bar\pi}.$$ Since the degree of a code is invariant under conjugacy of factor triples we deduce that $$\min\{N_\nu(\tilde\pi)\colon (X,Y,\pi)\cong(\bar X,\bar Y,\tilde\pi),\, \bar\pi \text{ is 1-block}\}=d_\pi.$$ Recalling that when $\pi$ is finite-to-one the degree and the class degree of $\pi$ are equal, the proof is complete. 
\end{proof}
Note that equality in Proposition~\ref{minbound} does not always hold. For example, consider the trivial factor code $\pi:\{0,1\}^\Z\ra \{0\}^\Z$. Then $c_\pi=1$; however, if $\left(\wt X,\{0\}^\Z,\tilde\pi\right)$ is a factor triple
conjugate to $\left(\{0,1\}^\Z, \{0\}^\Z,\pi\right)$ then $\A(\wt X)$ must be strictly greater than 1 and therefore
$N_\nu(\tilde\pi)>1$ for any ergodic measure $\nu$ on $Y$.

\begin{defn}Let $\pi:X\ra Y$ be a 1-block factor code from a SFT $X$ to a sofic shift
$Y$ and $\nu$ be an ergodic measure on $Y$. Let $\mu_1,\dots,\mu_n$
be invariant measures in the fibre $\piinv\{\nu\}$. The
\textsf{relatively independent joining}
$\tilde\mu=\mu_1\otimes\dots\otimes_{\nu}\mu_n$ of $\mu_1,\dots,\mu_n$
over $\nu$ is defined as follows: if $A_1,\dots,A_n$ are measurable
subsets of $X$ then
$$
\hat\mu(A_1\times\dots\times
A_n)=\int_Y\prod_{i=1}^{n}\E_{\mu_i}(\mathbf 1_{A_i}|\piinv\Balg_Y)\circ\piinv\,
d\nu.
$$
\end{defn}
Writing $p_i$ for the projection $X^n\ra X$ onto the $i$th
coordinate, it is shown that for $\hat\mu$-almost every $\hat x\in
X^n$, $\pi(p_i(\hat x))$ is the same for each $i$~\cite{rudolphB}.

We will use Theorem~\ref{orthogonal} below which is the main theorem from \cite{pqs} to prove a stronger theorem (Theorem~\ref{equivagonal}).

\begin{thm}\label{orthogonal}\cite[Theorem 1]{pqs}
Let $\pi:X\ra Y$ be a 1-block factor code from a 1-step SFT $X$ to a
sofic shift $Y$. Let $\nu$ be an ergodic measure on $Y$, and  two distinct ergodic measures $\mu_1$ and $\mu_2$ be measures of relative maximal entropy over $\nu$. Then $(\mu_1\otimes\mu_2)\{(u,v)\in X\times X\colon
u_0=v_0\}=0.$
\end{thm}

\begin{thm}\label{equivagonal}
Let $(X,Y,\pi)$ be a factor triple. Let $\nu$ be an ergodic measure on $Y$, and  two distinct measures $\mu_1$ and $\mu_2$ be ergodic measures of relative maximal entropy over $\nu$. Then 
$(\mu_1\otimes\mu_2)\{(u,v)\in X\times X\colon u\sim v\}=0$.
\end{thm}
\begin{proof}

First we show that, without loss of generality, we may assume $X$ is a 1-step SFT and $\pi$ is a 1-block factor code with a magic symbol. Suppose $(\wt X,\wt Y,\wt\pi)$ is a factor triple conjugate to $(X,Y,\pi)$ and $\phi:X\to\wt X$ is a conjugacy from $X$ to $\wt X$. By the proof of Proposition~\ref{prop:conjinv}(a) we have $$(\phi\times\phi)\{(u,v)\in X\times X\colon u\sim v\}=\{\left(\phi(u),\phi(v)\right)\in \wt X\times\wt X\colon \phi(u)\sim \phi(v)\}.$$ Moreover, the corresponding measure to $\mu_1\otimes\mu_2$ under the conjugacy $$\phi\times\phi:X\times X\to \wt X\times\wt X,$$ is $\wt\mu_1\otimes\wt\mu_2$ where $\wt\mu_1$ and $\wt\mu_2$ are corresponding measures to $\mu_1$ and $\mu_2$ under $\phi$. It follows that $$(\mu_1\otimes\mu_2)\{(u,v)\in X\times X\colon u\sim v\}=(\wt\mu_1\otimes\wt\mu_2)\{(\phi(u),\phi(v))\in\wt X\times\wt X\colon \phi(u)\sim \phi(v)\}.$$ Therefore by recoding, without loss of generality, we may assume $X$ is a 1-step SFT and $\pi$ is a 1-block factor code.

Let $u,\, v$  be in $X$, $n$ be in $\Z$, and $a$ be a symbol in $\A(X)$. Write
$u\overset{a}{\approx_n}v$ if $u\sim v$ and there are integers $m\leq n$ and $p\geq n$ such that 
\begin{equation*}
v_i=
\begin{cases}
u_i& -\infty<i<m,\, p<i<\infty
\\
a & i=n
\end{cases}
\end{equation*}
It is worth mentioning that the relation
$\overset{a}{\approx_n}$ is not in general an equivalence relation on $X$. We have
$$
\{(u,v): u\sim v\}=\bigcup_{a\in\A(X)}\bigcup_{n\in\Z}\big\{(u,v)\colon u\overset{a}{\approx_n}v\big\}.
$$
Now suppose, for a contradiction, that $(\mu_1\otimes\mu_2)\{(u,v): u\sim v\}>0$. There must be an integer $n$
and a symbol $a$ in $\A(X)$ such that
$$
(\mu_1\otimes\mu_2)\{(u,v)\colon u\overset{a}{\approx_n}v\}>0.
$$
By applying $(T\times T)^{-n}$, without loss of generality, we may assume $n=0$. Since we have
$$
\{(u,v)\colon u\overset{a}{\approx_0}v\}=\bigcup_{k\in\Z}\{(u,v)\colon u\overset{a}{\approx_0}v,\,u_{-k}^k,v_{-k}^k\text{ are routable through }a\text{ at time 0 }\},
$$
it follows that there is an integer $k$ so that
$$
(\mu_1\otimes\mu_2)\{(u,v)\colon u\overset{a}{\approx_0}v,\,u_{-k}^k,v_{-k}^k\text{ are routable through }a\text{ at time 0 }\}>0.
$$
 Considering only blocks of length $2k+1$, there must be blocks $A,\,B$ in $X$ such that 
\begin{equation}\label{d}
(\mu_1\otimes\mu_2)\big\{(u,v)\colon u\overset{a}{\approx_0}v,\,u_{[-k,\,k]}=A,\,v_{[-k,\,k]}=B \big\}>0.
\end{equation}
Both $A$ and $B$ are routable through $a$ at time $n$; i.e., there are blocks $A',\,B'$ in $\piinv_\text{b}(W)$ with $A'_{-k}=A_{-k}$, $B'_{-k}=B_{-k}$, $A'_0=B'_0=a$, $A'_{k}=A_k$, and $B'_k=B_k$. Let $G\subseteq \Z$ be the set $\{-k,\dots,k\}$. Basic properties of conditional expectation imply that
\begin{equation}\label{expe}
\begin{split}
\E_{\mu_1}\left(\mathbf 1_{[A]}|\piinv(\Balg_Y)\right)
&
=\E\left(\E_{\mu_1}\left(\mathbf 1_{\left._{-n}[A]\right.}|\piinv(\Balg_Y)\vee \sigma(X_{
G^c})\right)|\piinv(\Balg_Y)\right)
\\
&=\E\left(\E_{\mu_1}\left(\mathbf 1_{\left._{-n}[A']\right.}|\piinv(\Balg_Y)\vee \sigma(X_{
G^c})\right)|\piinv(\Balg_Y)\right)
\\
&=\E_{\mu_1}\left(\mathbf
1_{[A']}|\piinv(\Balg_Y)\right),
\end{split}
\end{equation}
where the second equality follows from Theorem~\ref{CUniformDist}.
Similarly we have
\begin{equation}\label{expec}
\E_{\mu_2}(\mathbf 1_{[B]}|\piinv\Balg_Y)=
\E_{\mu_2}(\mathbf 1_{[B']}|\piinv\Balg_Y).
\end{equation}
Let $D=\big\{(u,v)\colon u\overset{a}{\approx_0}v,\,u_{[-k,\,k]}=A,\,v_{[k,\,k]}=B \big\}$. Since
$D\subseteq\left._{-k}[A]\right.\times\left._{-k}[B]\right.$ we have
\begin{align*}
\mu(D)&\leq(\mu_1\otimes\mu_2)([A]\times [B])
\\
&=\int_Y\E_{\mu_1}(\mathbf 1_{[A]}|\piinv\Balg_Y)\E_{\mu_2}(\mathbf 1_{[B]}|\piinv\Balg_Y)\circ\piinv\,d\nu
\\
&= \int_Y\E_{\mu_1}(\mathbf
1_{[A']}|\piinv\Balg_Y)\E_{\mu_2}(\mathbf 1_{[B']}|\piinv\Balg_Y)\circ\piinv\, d\nu&\text{using}
\eqref{expe}\text{ and }\eqref{expec}
\\
&=
(\mu_1\otimes\mu_2)([A']\times [B'])
\\
&=0
\end{align*}
where the last equality follows from Theorem~\ref{orthogonal} since
$A'_0=B'_0=a$. This contradicts Equation \eqref{d}.
\end{proof}
The following theorem is a more general case of Theorem~\ref{thm:final}.
\begin{thm}\label{bound}
Let $(X,Y,\pi)$ be a factor triple and $\nu$ be an ergodic measure on $Y$. The number of ergodic measures of relative maximal entropy over $\nu$ is at most $|\C(\nu)|$.
\end{thm}
\begin{proof}
Let $(\wt X,\wt Y,\pi)$ be a factor triple conjugate to $(X,Y,\pi)$ under conjugacies $\phi:X\to \wt X$ and $\psi:Y\to\wt Y$. By Observation~\ref{obs:numberofMRME}, the number of ergodic measures of relative maximal entropy over $\nu$ and $\nu\circ\psi^{-1}$ is the same. Moreover, by Proposition~\ref{prop:conjinv}(a), for each $y$ in $Y$ we have $|\C_\pi(y)|=|\C_{\wt\pi}(\psi(y))|$. Therefore by recoding, without loss of generality, we may assume $X$ is a 1-step SFT and $\pi$ is a 1-block factor code.

Suppose, for a contradiction, that there are $n>|\C(\nu)|$ ergodic measures $\mu_1,\dots,\mu_n$ on $X$ of relative maximal entropy over $\nu$. Form the relatively independent joining $\hat\mu$ on
$X^n$ of the measures $\mu_1,\dots,\mu_n$. Note that for
$\hat\mu$-a.e. $\hat x\in X^n$, $\pi(p_i(\hat x))$ has $|\C(\nu)|$ transition classes over it. The assumption $n>|\C(\nu)|$ implies that for $\hat\mu$-a.e. $\hat
x=(x_1,\dots,x_n)\in X^n$ there are distinct $i,j$ such that $p_i(\hat
x)\thicksim p_j(\hat x)$; i.e.,
$$
\hat\mu\left(\bigcup_{1\leq i<j\leq n}\{\hat x=(x_1,\dots,x_n)\colon
x_i\thicksim x_j\}\right)=1.
$$
At least one of the sets $S_{i,j}=\{(x_1,\dots,x_n)\colon x_i\thicksim
x_j\}$ must have positive $\hat\mu$-measure. It follows that
\begin{equation*}
\begin{split}
0&<\hat\mu(S_{i,j})
\\
&=\hat\mu(\{(x_1,\dots,x_n),~x_i\thicksim x_j\})
\\
&=
(\mu_i\otimes\mu_j)\{(u,v)\colon u\thicksim v\}.
\end{split}
\end{equation*}
This contradicts Theorem~\ref{equivagonal}.
\end{proof}

\section{Open Questions}
Let $X$ be a one-sided topologically mixing shift of finite type. An invariant measure $\mu$ on $X$ is a \textsf{Gibbs measure} corresponding to $f\in C(X)$ if there are constants $C_1,~C_2>0$ and $P>0$ such that $$C_1\leq\frac{\mu([x_0x_1\dots x_{n-1}])}{\exp\left(-Pn+(S_nf)(x)\right)}\leq C_2$$ for every $x\in X$ and $n\geq 1$, where $(S_nf)(x)=\sum_{k=0}^{n-1} f(T^k(x))$.

Walters \cite{waltersConv,waltersRegula} introduces a class Bow$(X)$ of functions that contains the functions with summable variation, all of which have unique equilibrium states. Let 
$$\text{var}_n(f)=\sup\{|f(x)-f(y)|\colon x,~y\in X,~x_i=y_i \text{ for all } 0\leq i\leq n-1\}.$$ 
Then Bow$(X,T)=\{f\in C(X)\colon \sup_{n\geq 1} \text{var}_n(S_nf)<\infty\}.$
\begin{thm}\label{thm:equilib}\cite[Theorem 2.16]{waltersConv}
Let $f\in\text{Bow}(X)$. Then $f$ has a unique equilibrium state $\mu$ which is a Gibbs measure. 
\end{thm}

The relative version of this result is an open question. We make the following conjecture.
\begin{conj}
Let $(X,Y,\pi)$ be a factor triple and $\nu$ be a fully supported ergodic measure on $Y$. For any function $f\in \text{Bow}(X)$ the number of ergodic measures of maximal pressure of $f$ in the fibre $\pi^{-1}\{\nu\}$ is at most $c_\pi$.
\end{conj}


\begin{thebibliography}{10}

\bibitem{black}
D.~Blackwell.
\newblock The entropy of functions of finite-state {M}arkov chains.
\newblock In {\em Transactions of the first {P}rague conference on information
  theory, {S}tatistical decision functions, random processes held at {L}iblice
  near {P}rague from {N}ovember 28 to 30, 1956}, pages 13--20. Publishing House
  of the Czechoslovak Academy of Sciences, Prague, 1957.

\bibitem{boyle05}
M.~Boyle.
\newblock Putnam's resolving maps in dimension zero.
\newblock {\em Ergodic Theory Dynam. Systems}, 25(5):1485--1502, 2005.
\newblock See erratum at
  http://www-users.math.umd.edu/~mmb/papers/degreeerratum.

\bibitem{BoyTun}
M.~Boyle and S.~Tuncel.
\newblock Infinite-to-one codes and {M}arkov measures.
\newblock {\em Trans. Amer. Math. Soc.}, 285(2):657--684, 1984.

\bibitem{BurRos}
C.~J. Burke and M.~Rosenblatt.
\newblock A {M}arkovian function of a {M}arkov chain.
\newblock {\em Ann. Math. Statist.}, 29:1112--1122, 1958.

\bibitem{bs}
R.~Burton and J.~E. Steif.
\newblock Non-uniqueness of measures of maximal entropy for subshifts of finite
  type.
\newblock {\em Ergodic Theory Dynam. Systems}, 14(2):213--235, 1994.

\bibitem{GatPer}
D.~Gatzouras and Y.~Peres.
\newblock The variational principle for {H}ausdorff dimension: a survey.
\newblock In {\em Ergodic theory of {${\bf Z}^d$} actions ({W}arwick,
  1993--1994)}, volume 228 of {\em London Math. Soc. Lecture Note Ser.}, pages
  113--125. Cambridge Univ. Press, Cambridge, 1996.

\bibitem{GatPerexp}
D.~Gatzouras and Y.~Peres.
\newblock Invariant measures of full dimension for some expanding maps.
\newblock {\em Ergodic Theory Dynam. Systems}, 17(1):147--167, 1997.

\bibitem{isr}
R.~B. Israel.
\newblock {\em Convexity in the theory of lattice gases}.
\newblock Princeton University Press, Princeton, N.J., 1979.
\newblock Princeton Series in Physics, With an introduction by Arthur S.
  Wightman.

\bibitem{kel}
G.~Keller.
\newblock {\em Equilibrium states in ergodic theory}, volume~42 of {\em London
  Mathematical Society Student Texts}.
\newblock Cambridge University Press, Cambridge, 1998.

\bibitem{LanfRuel69}
O.~E. Lanford, III and D.~Ruelle.
\newblock Observables at infinity and states with short range correlations in
  statistical mechanics.
\newblock {\em Comm. Math. Phys.}, 13:194--215, 1969.

\bibitem{LedWal}
F.~Ledrappier and P.~Walters.
\newblock A relativised variational principle for continuous transformations.
\newblock {\em J. London Math. Soc. (2)}, 16(3):568--576, 1977.

\bibitem{LedYo}
F.~Ledrappier and L.-S. Young.
\newblock The metric entropy of diffeomorphisms. {II}. {R}elations between
  entropy, exponents and dimension.
\newblock {\em Ann. of Math. (2)}, 122(3):540--574, 1985.

\bibitem{lm}
D.~Lind and B.~Marcus.
\newblock {\em An introduction to symbolic dynamics and coding}.
\newblock Cambridge University Press, Cambridge, 1995.

\bibitem{MarPetWill}
B.~Marcus, K.~Petersen, and S.~Williams.
\newblock Transmission rates and factors of {M}arkov chains.
\newblock In {\em Conference in modern analysis and probability ({N}ew {H}aven,
  {C}onn., 1982)}, volume~26 of {\em Contemp. Math.}, pages 279--293. Amer.
  Math. Soc., Providence, RI, 1984.

\bibitem{parry}
W.~Parry.
\newblock Intrinsic {M}arkov chains.
\newblock {\em Trans. Amer. Math. Soc.}, 112:55--66, 1964.

\bibitem{karlinfo}
K.~Petersen.
\newblock Information compression and retention in dynamical processes.
\newblock In {\em Dynamics and randomness ({S}antiago, 2000)}, volume~7 of {\em
  Nonlinear Phenom. Complex Systems}, pages 147--217. Kluwer Acad. Publ.,
  Dordrecht, 2002.

\bibitem{pqs}
K.~Petersen, A.~Quas, and S.~Shin.
\newblock Measures of maximal relative entropy.
\newblock {\em Ergodic Theory Dynam. Systems}, 23(1):207--223, 2003.

\bibitem{rudolphB}
D.~J. Rudolph.
\newblock {\em Fundamentals of measurable dynamics}.
\newblock Oxford Science Publications. The Clarendon Press Oxford University
  Press, New York, 1990.
\newblock Ergodic theory on Lebesgue spaces.

\bibitem{ruelle}
D.~Ruelle.
\newblock Statistical mechanics on a compact set with {$Z^{v}$} action
  satisfying expansiveness and specification.
\newblock {\em Trans. Amer. Math. Soc.}, 187:237--251, 1973.

\bibitem{shannon}
C.~E. Shannon.
\newblock A mathematical theory of communication.
\newblock {\em Bell System Tech. J.}, 27:379--423, 623--656, 1948.

\bibitem{shin}
S.~Shin.
\newblock Measures that maximize weighted entropy for factor maps between
  subshifts of finite type.
\newblock {\em Ergodic Theory Dynam. Systems}, 21(4):1249--1272, 2001.

\bibitem{walters}
P.~Walters.
\newblock {\em An introduction to ergodic theory}, volume~79 of {\em Graduate
  Texts in Mathematics}.
\newblock Springer-Verlag, New York, 1982.

\bibitem{wal}
P.~Walters.
\newblock Relative pressure, relative equilibrium states, compensation
  functions and many-to-one codes between subshifts.
\newblock {\em Trans. Amer. Math. Soc.}, 296(1):1--31, 1986.

\bibitem{waltersConv}
P.~Walters.
\newblock Convergence of the {R}uelle operator for a function satisfying
  {B}owen's condition.
\newblock {\em Trans. Amer. Math. Soc.}, 353(1):327--347 (electronic), 2001.

\bibitem{waltersRegula}
P.~Walters.
\newblock Regularity conditions and {B}ernoulli properties of equilibrium
  states and {$g$}-measures.
\newblock {\em J. London Math. Soc. (2)}, 71(2):379--396, 2005.

\end{thebibliography}
\end{document}